\documentclass{amsart}
\usepackage{amsmath,amsfonts,amsthm,amssymb,indentfirst,epic,url,graphics}
\usepackage{hyperref}
\usepackage{fancyhdr}%% added 
\usepackage{graphicx}
\usepackage{amsaddr}

\newtheorem{theorem}{Theorem}

\newtheorem{corollary}[theorem]{Corollary}
\newtheorem{proposition}[theorem]{Proposition}
\theoremstyle{definition}

\newtheorem{remark}[theorem]{Remark} %Rich 9/14

\begin{document}

\begin{center}
\texttt{A version of this article was submmitted to {\it Statistics and Probability Letters} March 5, 2015.} \\[.5em]
{\TeX}ed \today 
\vspace{2em}
\end{center}

\title%
{The Random Division of the Unit Interval\\
and the Approximate -1 Exponent\\
in the Monkey-at-the-Typewriter Model of Zipf's Law}

\author{ Richard Perline}
\address{Flexible Logic Analytics, 34-50 80th Street, Jackson Heights, NY  11372}
\email{richperline@gmail.com}
\maketitle

%% following from:   https://www.sharelatex.com/learn/Headers_and_footers
\pagestyle{fancy}
\fancyhf{}
\rhead{Richard Perline}
\lhead{The Random Division of the Unit Interval ...}
\cfoot{Page \thepage}

\begin{abstract}
We show that the exponent in the inverse power law of word frequencies for the monkey-at-the-typewriter model of Zipf's law will tend towards $-1$
under broad conditions
as the alphabet size increases to infinity and the letter probabilities are specified as the values from a random division of the unit interval. This is proved utilizing a strong limit theorem for log-spacings due to Shao and Hahn.  
\end{abstract}

\section{Introduction}

By the monkey-at-the-typewriter model we mean the random scheme  generating words defined as a sequence of letters  terminating with a space character.
Using the simple case of $K\ge 2$ {\it equal} letter probabilities plus one space character and an independence assumption, G. A. Miller (1957;  Miller and Chomsky, 1963) highlighted a somewhat hidden aspect of Mandelbrot's (1953, 1954a, 1954b) work on Zipf's law by 
showing how this scheme generates an inverse power law for word frequencies mimicking Zipf's law for natural languages.  Miller also 
observed empirically that the exponent of the power law  in his model was close to $-1$ for his numeric example with $K=26$ letters. 
An exponent value near  $-1$ is especially interesting because it is an iconic feature of empirical word frequency data for most Western languages,  as Zipf  (1935, 1949)
amply demonstrated.   In other words, not only does Miller's simple model generate an inverse power law, but by letting
the alphabet size $K$ be sufficiently large, it also approximates  the same parameter value commonly seen with real word frequency data. 

The power law behavior of the monkey model with {\it unequal} letter probabilities is substantially more complicated to analyze.
Utilizing tools from analytic number theory, Conrad and Mitzenmacher (2004) have provided the first  fully rigorous analysis of the monkey model power law in this general case.  
They did not comment about Miller's remark concerning a power law exponent close to $-1$.
  Our main objective in this paper is to analyze the behavior of the exponent.  We do this by
specifying the letter probabilities as the spacings from a random division of the unit interval and then make use of a strong limit theorem 
for log-spacings due to Shao and Hahn (1995).  This theorem will let us show that the approximate $-1$ exponent is  an almost universal
parameter that results from the asymptotics of log-spacings as $K\to\infty$ and where the spacings are generated from a broad class of distributions - for example, any distribution with a bounded density function on the unit interval will satisfy their theorem.  Our idea for studying spacings in connection with 
the approximate $-1$ exponent was first given in Perline (1996, section 3), but  the argument there should now be viewed as just a small step in the direction we follow here.

In Section 2 we demonstrate the power law behavior of the monkey model with arbitrary letter probabilities and
obtain a formula for the exponent $-\beta$.  
The derivation of this in Conrad and Mitzenmacher is somewhat intricate,    
but they also provide an instructive Fibonacci example  (Mitzenmacher, 2004, gives the same example)   that motivated us down
a simpler path.
Our application of 
Csisz\'ar's (1969) approach below requires only elementary methods and also brings into focus 
Shannon's (1948) ingenious difference equation logic as he used it originally to define the capacity of a noiseless communication channel.
This combinatorial scheme employs precisely the same line of thinking as in Conrad and Mitzenmacher's  Fibonacci example  and as Mandelbrot used in his work on Zipf's law.
Bochkarev and Lerner (2012) have also provided an elementary analysis of the power law behavior using the Pascal pyramid.

Section 3 contains our main result showing conditions such that $-\beta\approx -1$ for the monkey model  using the Shao and Hahn (1995) asymptotics for log-spacings. Throughout this article, all logarithms use the natural base $e$ unless the radix is explicitly indicated.

\section{The Monkey Model and Its Power Law Exponent $-\beta$}

\subsection{Defining the Model}

Consider a keyboard with an alphabet of $K\ge 2$ letters $\{L_1\dots,L_K\}$ and space character $S$.  Assume that the letter characters are 
labelled such that their 
corresponding non-zero probabilities are rank ordered as  $q_1\ge q_2, \dots\ge  q_K$. The space character has probability $s$, so that $\sum_{i=1}^K q_i + s =1$.
A word is defined as any sequence of non-space letters terminating with a space.  
A word $W$ of exactly $n$ letters is a string such as  $W=L_{i_1}L_{i_2}\dots L_{i_n}S$ and    has a probability of the form $P(W)=P=q_{i_1}q_{i_2}\dots q_{i_n}s$ because letters are struck independently.  The space character with no preceding letter character will be considered a word of length zero.

Write  the rank ordered sequence of descending word probabilities  in the ensemble of all possible words of finite length as
$P_1=s >  P_2\ge  \dots P_r  \ge \dots$ ($P_1=s$ is always the first and largest word probability.)
   Break ties for words with
equal probabilities by alphabetical ordering, so that each word probability has a unique rank $r$. 
In Miller's (1957) version of the monkey model, all letter probabilities have the same value, $q_1=\dots =q_K=(1-s)/K$ and so words of the same length
have the same probability.

It will be convenient to
work with the word {\it base values}, $B=P/s$, which are  simply the product of the letter probabilities without the space probability.   
Write $B_r$ for  $P_r/s$. 
Figure 1(a) shows a log-log plot of $\log B_r$ vs. $\log r$ for the particular example used by Miller (1957).  In his example the space character was given the value
$s=.18$ (similar to what is seen for English text) and so the  $K=26$ letters each are given  the identical probability $(1-.18)/26\approx .0315$.  Figure 1(a) plots all  base values  for words of 4 or fewer letters, i.e., 
$\sum_{j=0}^4 26^j = {26^5 -1 \over 25}=475,255$ values.  The horizontal lines show the steps in  a power law for Miller's model.  These steps arise
from the fact that 
 there will be $26^j$ words of length $j$ letters, all with the  same base value $B=(.82/26)^j$, but with $26^j$ different (though consecutive) ranks. 
The other log-log plots in Figure 1 are based on using {\it unequal} letter probabilities chosen as the spacings of a uniform distribution (Figure 1(b))  and from the
beta(3,2) distribution  (Figure 1(c)) with density $f(x)=x^{3-1}(1-x)^{2-1}/B(3,2)$.   The reason
we see an approximate $-1$ slope for the graph in 1(a) will  be explained in Section 2.2 and the explanation for the graphs of Figure 1(b) and Figure 1(c) will be given in Section 3.

\subsection{The Special Case of Equal Letter Probabilities}

\begin{figure}
\centering
\includegraphics[width=6.5in]{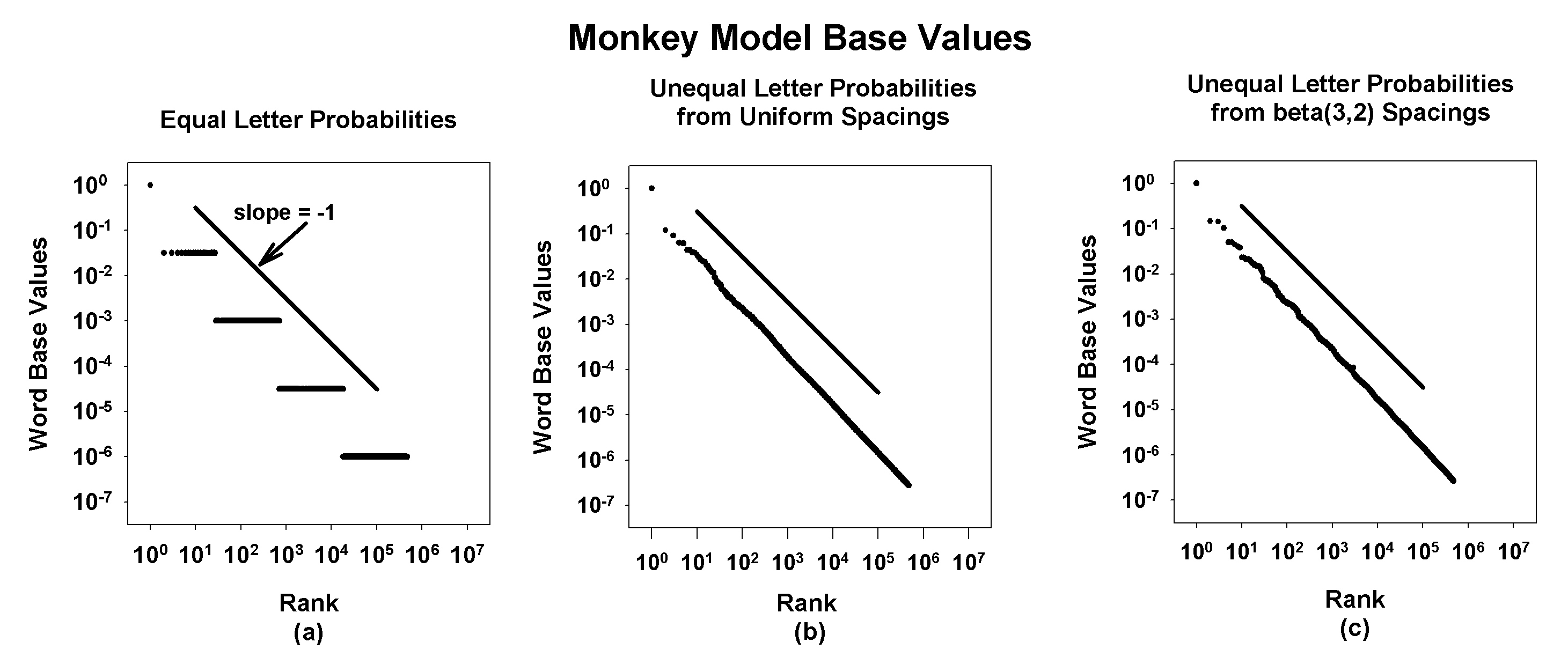}
\caption{Monkey-at-the-typewriter word base values where the $K=26$ letter probabilities have been specified as: (a) all equal; (b) the spacings from a uniform distribution; (c) the spacings
from a beta(3,2) distribution.  Note the approximate $-1$ slopes in all three cases. }
\end{figure}

The  log-log linear appearance of the equal letter probability model in  Figure 1(a) is easily explained, as Miller (1957) showed, and
he also drew attention to the approximate $-1$ log-log slope.  
We will put our discussion  in the framework of Conrad and Mitzenmacher's (2004)  
analysis of the monkey model where they 
carefully defined power law behavior for  monkey word probabilities  as the situation where the inequality 
\begin{equation}
A_1 r^{-\beta} \le P_r \le A_2 r^{-\beta}, \nonumber
\end{equation} 
holds for sufficiently large rank $r$ and  constants
$0<A_1 \le  A_2$, $\beta>0$. 
It will also prove convenient to state this inequality in an equivalent form where the rank $r$ is bounded
in terms of $P_r$:
\begin{equation}
A_1^{\prime} P_r^{-1/\beta}=A_1^{-1/\beta} P_r^{-1/\beta}\le  r \le A_2 ^{-1/\beta} P_r^{-1/\beta}=A_2^{\prime} P_r^{-1/\beta}, \nonumber
\end{equation} 
or in the form that we  use here in terms of  base values $B_r=P_r/s$,  
\begin{equation}
C_1  B_r^{-1/\beta}\le  r \le C_2  B_r^{-1/\beta}.  
\end{equation}

In the equal letter probability model of Miller (1957) and Miller and Chomsky (1963), the base value $B_r=q_1^j$ for some $j$.  
Then there are $K^j$ words with base values equal to $B_r$  and  $\sum_{i=0}^{j-1} K^i $ words with base values strictly smaller than $B_r$.
Therefore  the rank $r$ of any base value equal to $B_r$ satisfies the inequality
\begin{equation}
{1\over K} K^j<\sum_{i=0}^{j-1} K^i  + 1 \le  r \le \sum_{i=0}^j K^i <{K\over K-1} K^j.
\end{equation}
Now
\begin{equation}
 K^j=\Bigl(q_1^{\log_{q_1} K}\Bigr)^j =(q_1^j)^{\log_{q_1} K}=B_r ^{\log_{q_1} K}, \nonumber
\end{equation}
therefore, the outer bounds of the inequality in (2) can be written as
\begin{equation}
{1\over K}B_r ^{\log_{q_1} K}  <  r <{K\over K-1} B_r ^{\log_{q_1} K}. \nonumber
\end{equation}
This is in the form of the inequality (1) with constants $C_1={1 /  K}$, $C_2={K ( K-1)^{-1}}$, and $-{1 /  \beta}=\log_{q_1} K$, 
demonstrating the power law behavior of  Miller's model. 

It is worth pointing out that we can adopt a very simple view of things by working in log-log scales.
Inequality (1) implies the weaker asymptotic conditions  $\log r \sim -\log B_r/\beta$
and $\log B_r \sim -\beta \log r$
 as $r\to \infty$.   We'll refer to this as  {\it log-log linear} behavior and note that with this view,
$\log C_1$ and $\log C_2$ are asymptotically negligible so that only the parameter $\beta$ becomes important.
With Miller's model, a plot of $\log B_r$ by $\log r$ produces a step-function approximation to a line with
slope $-\beta=1 /  \log_{q_1} K$.  

Miller then made the additional  observation that with the parameters used to create Figure 1(a),  his model gives $-\beta\approx -1.06$, close to $-1$.  He did not go further than this, but it is easy
to see that
\begin{eqnarray}
 -\beta&=&{1 \over \log_{q_1} K}  \nonumber   \\  
&=&{\log q_1 \over \log K} \nonumber  \\
&=&{1\over \log K} \log \Bigl({1-s\over K}\Bigr) \nonumber    \\
&=&{ \log(1-s)\over \log K} -1 \to -1, \nonumber
\end{eqnarray}
as $K\to \infty$.  Consequently, for sufficiently large $K$,  the exponent in his model will be close to  $-1$ and so plots of $\log B_r$ vs. $\log r$ will exhibit
log-log linear behavior with a slope in the vicinity of this value.  

\subsection{The General Case of Unequal Letter Probabilities}

We begin by representing the $K$ letter probabilities
as powers of the maximum letter probability value $q_1$ so that 
$q_1=q_1^{\alpha_1} \ge  q_2=q_1^{\alpha_2} \cdots\ge q_K=q_1^{\alpha_K}$,
where $1=\alpha_1 \le \alpha_2 \cdots\le \alpha_K$ are arbitrary real numbers.   
Then for every base value for a word of  length $n \ge 1$,   $B=q_1 ^{\alpha_{i_1}+\cdots +\alpha_{i_n}}$.
The largest base value for the null word of length 0 consisting of the space character alone has $B_1=1=q_1^0$. 
Applying a radix-$q_1$ logarithm to any base value gives a sum $\log_{q_1} B=\alpha_{i_1}+\cdots+\alpha_{i_n}$
for any word of $n\ge 1$ letters and 0 for the null word. 

Now we introduce Csisz\'ar's formulation of Shannon's recursive counting logic.
 For any real $t\ge 0$, define $N(t)$ as the number of monkey words that have radix-$q_1$ log base values, $\log_{q_1} B$,  in the
half open interval $(t-\alpha_1, t]$.  
By construction, $\alpha_1=1$, so this interval is 
 $(t - 1, t]$. 
Any word of  of $n\ge 1$ letters must begin with one of the letters $L_1, L_2,\dots, L_K$.
Among the $N(t)$ words with log base values in $(t - 1, t]$, the number
of words beginning with $L_i$ is 
$N(t - \log_{q_1} q_i)=N(t- \alpha_i)$.  Since this is true for $i=1,\dots,K$,   
the recursion
\begin{equation}
N(t)=\sum_{i=1}^K N(t -\alpha_i)\ \ {\rm for}\ \ t \ge \alpha_1= 1, 
\end{equation} 
holds with the additional  conditions $N(t)=0$ if $t<0$ and $N(t)=1$ if $0\le t <1$.
Csisz\'ar (his  Proposition 1.1)  then proves:
\begin{theorem}[]\label{CsiszarTheorem}
There is a positive constant $0<b< 1$, such that $N(t)$ has the bounds
\begin{equation}
bR_0^{t}< N(t) \le R_0^{t} \ \ {\rm for}\  t\ge 0, 
\end{equation}
where $R_0>1$ is the unique positive root  of the equation 
\begin{eqnarray}
\sum_{i=1}^K X^{-\alpha_i}=1.
\end{eqnarray}
\end{theorem}

\begin{proof}
Csisz\'ar first 
 establishes that (5) has a unique positive root $R_0>1$ and then gives this induction argument
to confirm the bounds in (4).  
 Recursively define a sequence of positive numbers $b_i$ by: 
\begin{equation}
b_{i+1}=b_i \sum_{{\alpha_j}\le i} R_0^{-\alpha_j}, \ \  i=1, 2, \dots \ {\rm and}\  b_1=R_0^{-\alpha_1}=R_0^{-1}. \nonumber
\end{equation}
The inequality 
\begin{equation}
b_i R_0^{t} < N(t)\le R_0^{t} \ \  {\rm if} \ 0\le t <i
\end{equation}  holds for $i=1$.  Now  assume it also holds for  the integers $1,2,\dots,i$.
Then for $i\le t  < i+1$,
\begin{equation}
b_{i+1}R_0^{t}=\big(b_i\sum_{\alpha_j\le i}R^{-\alpha_j}\big)R_0^t  \le \sum_{{\alpha_j}\le t} b_i R_0^{t -\alpha_j} < \sum_{j=1}^K N(t -\alpha_j)
 \le \sum_{j=1}^K R_0^{t -{\alpha_j}}=R_0^{t}, \nonumber
\end{equation} 
where we use the facts that $t-\alpha_j<i$ and  that $\sum_{j=1}^K R_0^{-\alpha_j}=1$.
Since $\sum_{j=1}^K N(t -\alpha_j)=N(t)$ by (3), we have shown that
if the inequality (6) holds for $i$, it holds for $i+1$, as well.
Lastly,  the  $b_i$ are non-increasing and since $b_1=1/R_0<1$,  they  converge after a finite number of steps to  some $0<b < 1$,
completing the induction proof.  
\end{proof}

We can now find bounds on the rank $r$ of a base value $B_r$ by  
calculating a cumulative sum involving $N(\log_{q_1} B_r)$.   This is the basic idea Mandelbrot used in the context of his information-theoretic models,
where instead of  $\log_{q_1} B_r$ he has a cost value $C_r$
(Brillouin (1956) provides a helpful discussion.)
Let $n=\lfloor \log_{q_1} B_r\rfloor \ge 0$ be the greatest integer contained in $\log_{q_1} B_r$.
 Then the number of words with radix-$q_1$  log base values $\le \log_{q_1}  B_r$ (equivalently, having base values
$\ge B_r$)
  is given by 
\begin{equation}
N_{\rm cum} (\log_{q_1} B_r)= N(\log_{q_1} B_r) + N(\log_{q_1} B_r -1) + \cdots + N(\log_{q_1} B_r -n). \nonumber 
\end{equation}
This means that the rank $r$ of any base value $\ge B_r$ satisfies $r\le N_{\rm cum} (\log_{q_1} B_r)$, so  
by Theorem 1, 
\begin{eqnarray}
r&\le& N_{\rm cum} (\log_{q_1} B_r)  \nonumber  \\ 
&\le&  R_0^{\log_{q_1} B_r} + R_0^{\log_{q_1}B_r-1}+\cdots + R_0^{\log_{q_1} B_r-n}  \nonumber \\
&<& \Bigl({R_0\over R_0-1}\Bigr)R_0^{\log_{q_1} B_r}. 
\end{eqnarray}
On the other hand, since $B_r < B_r /q_1$, we must have $r> N_{\rm cum}(\log_{q_1} B_r-1)$ and so Theorem 1 provides
a lower bound on $r$:
\begin{eqnarray}
r&>& N_{\rm cum}(\log_{q_1} B_r-1)\nonumber   \\
&>& N(\log_{q_1} B_r-1)\nonumber \\
&> &b R_0^{\log_{q_1} B_r-1}\nonumber \\
&=& {b\over R_0} R_0^{\log_{q_1} B_r},
\end{eqnarray}
for some $0<b< 1$.  Combining (7) and (8), 
\begin{eqnarray}
C_1  B_r^{\log_{q_1} R_0}={b\over R_0} R_0^{\log_{q_1} B_r}<r< \Bigl({R_0\over R_0-1}\Bigr)R_0^{\log_{q_1} B_r}=C_2 B_r^{\log_{q_1} R_0}, \nonumber 
\end{eqnarray}
which demonstrates power law behavior by the Conrad-Mitzenmacher inequality criterion in (1) with $-1/\beta=\log_{q_1} R_0$ .  
 Hence a plot of $\log B_r$ versus $\log r$ will 
produce an asymptotically log-log linear graph with slope $-\beta=1/\log_{q_1} R_0$.  Note that $-\beta$ can be written
in several different ways:  $-\beta=1/\log_{q_1} R_0=\log {q_1}/\log R_0= \log_{R_0} q_1$.
As Bochkarev and Lerner (2012) point out, $\beta$  is also the solution to the equation
$q_1^{1/\beta}+q_2^{1/\beta}+\dots + q_K^{1/\beta}=1$.  

In Miller's case with $q_1=\cdots=q_k=(1-s)/K$, we obtain $R_0=K$ as the solution to
$\sum_{i=1}^K X^{-1}=KX^{-1}=1$,  giving $-\beta=1 / \log_{q_1} K$,  as seen above.   
Let us also record here that $-\beta={1/\log_{q_1} R_0}<-1$.  This is true
because $q_1 + q_1^{\alpha_2} +\cdots + q_1^{\alpha_K}=1-s<1 = R_0^{-1} + R_0^{-\alpha_2}+\cdots+R_0^{-\alpha_K}$, and it follows that
$q_1<R_0^{-1}$ and that  ${1 /  \log_{q_1} R_0}={\log  q_1 / \log R_0 }< -1$.

We turn now to examine conditions where
$-\beta$ approaches $-1$ for the general case of unequal letter probabilities. 

\section{$-\beta \approx -1$ with Large $K$ from the Asymptotics of Log-Spacings}

To produce the graph of Figure 1(b), 26 letter probabilities for the monkey typewriter keys were generated using the classic ``broken stick'' model.    We drew a sample of $K-1=25$ uniform random variables $X_1, X_2, \dots, X_{K-1}$  from $[0,1]$.  Denote their order statistics as $X_{1:K-1}\ge X_{2:K-1}\dots \ge X_{K-1:K-1}$.  The interval was partitioned into $K$ mutually exclusive and exhaustive segments by the
$K$ {\it spacings} $D_i$ defined   as the differences between successive uniform order statistics:   $D_1=1-X_{1:K-1}$,
 $D_i=X_{i-1:K-1} - X_{i:K-1}$ for $2\le i \le K-1$, and $D_{K}=X_{K-1:K-1}$, so that $\sum_{i=1}^K D_i=1$.   Write the order statistics of the spacings themselves   
as  $D_{1:K}\ge D_{2:K}\ge\dots\ge D_{K:K}$. 
After obtaining 
a sample of $25$ uniform random variables from $[0,1]$ the letter probabilities 
 were specified as
$q_i=.82 D_{i:26}$.  The factor .82 reflects the fact that .18 was used for the probability of the 27th typewriter key, the  space character, matching the example of Miller (1957).   These letter probabilities were then used to generate the base values of words, and we then graphed the 475,255 largest base values in Figure 1 (b), corresponding to the 475,255 largest base values graphed in Figure 1(a),  where Miller's equal probability model was used.
We mention that 
Good (1969, p. 577), Bell et al (1990, p. 92) and other researchers have noted  that the distribution of letter frequencies  in English is
well approximated by the expected values of uniform spacings.

Figure 1(c) was generated by the same process except that the random variables $X_1,\dots,X_{25}$ were drawn
from the  beta distribution ${\rm beta}(3,2)$ with pdf $f(x)=x^{3-1}(1-x)^{2-1}/B(3,2)$.   
To understand why $-\beta\approx -1$ in Figures 1(b) and 1(c),  we 
make use of a  strong limit theorem for the logarithms of spacings due to  Shao and Hahn (1995, Corollary 3.6; 
see also Ekstr\"om, 1997, Corollary 4).   {\it Almost sure convergence} is denoted $\xrightarrow{a.s.}$ and  {\it almost everywhere} is abbreviated $a.e.$.

\begin{theorem}
Let $X_1, \dots, X_{K-1}$ be $K-1$ i.i.d.  random variables with on $[0,1]$ with cumulative distribution function $H(x)$.
Define $H^{-1}(y)={\rm inf}\{x: H(x)>y\}$.   Assume $f(x)= (d/dx)H^{-1}(x)$ exists $a.e.$   If 
there is an $\epsilon_0>0$ such that $f(x)\ge \epsilon_0$ $a.e.$, then  
\begin{equation}
{1\ \over K} \sum_{i=1}^{K} \log\Bigl(K D_i\Big) \xrightarrow{a.s.} \int_0^1 \log f(x) dx - \lambda \ \ {\rm as}\ K \to\infty,   
\end{equation} 
where $\lambda=.577\dots$ is Euler's constant.   
\end{theorem}

\begin{remark}
Shao and Hahn call $\int_0^1 \log f(x)dx$ the generalized entropy of $H$.
If $H(x)$ has a density function $(d/dx) H(x)=h(x)$, then $f(x)=(d/dx) H^{-1}(x)=1/h(H^{-1}(x))$ and a change of variable
shows that $\int_0^1 \log  f(x) dx=-\int_0^1 h(x) \log h(x) dx$, which is the differential entropy of $h(x)$.  Note that when a density exists, the bound
$f(x)\ge \epsilon_0>0$ imples that $h(x)\le 1/\epsilon_0<\infty$, i.e., this theorem holds for random variables on $[0,1]$ with densities bounded  
away from infinity.
\end{remark}

\begin{corollary}
With radix-$K$ logarithms, Theorem 2 yields 
\begin{equation}
{ \sum_{i=1}^{K} \log_K D_i \over K}  \xrightarrow{a.s.}   -1 \ {\rm as}\ K \to\infty. \nonumber  
\end{equation}
\end{corollary}
\begin{proof}
Rewriting the left side of the limit (9)  as $\log K + {\sum_{i=1}^{K} \log D_i  /  K}$ and
dividing both sides by $\log K$ gives
\begin{equation}
{\log K \over \log K} + { \sum_{i=1}^{K} \log D_i \over K\log K}  \xrightarrow{a.s.}  {{\int_0^1 \log f(x) dx}\over \log K} -  {\lambda \over \log K} \ \ {\rm as}
\  K\to\infty. \nonumber
\end{equation} 
The terms on the right of the limit $\to 0$ as $K\to\infty$.
Expressing logarithms using radix-$K$ and subtracting $1=\log K / \log K$ from both sides of the limit  completes the proof.
\end{proof}

Figures 1(b) and 1(c) were generated using spacings from the 
 uniform and beta(3,2) distributions where 
we selected letter probabilities from the interval $[0,c]$
by specifying 
$q_i=c*D_{i:K}$ ($0<c<1$).    Our choice for this example was $c=.82$ corresponding to what Miller used.  
Defining  $\overline{\mu}_K=\sum_{i=1}^K {\log_K q_i }/K$, 
we therefore have 
\begin{equation}
 { \overline{\mu}_K} = {\log_K c} +  {\sum_{i=1}^K \log_K D_{i:K} \over K} \xrightarrow {a.s.} -1,   \nonumber
\end{equation}
since ${\log_K c  =\log c/\log K}$ $\to 0$ as $K\to\infty$.  Consequently, $c\in(0,1)$ is asymptotically negligible for our purposes.  

It now remains to explain how $\overline{\mu}_K\xrightarrow{a.s.}-1$ relates to the log-log slope 
$-\beta={1 / \log_{q_1} R_0}=\log_{R_0} q_1$ obtained from the Shannon-Csisz\'ar-Mandelbrot difference
equation calculation of Section 2.  Since we saw that $-\beta <-1$,   if  we can now show that $\overline{\mu}_K \le -\beta$,  then for sufficiently large $K$, $-\beta$ will be forced close to $-1$.

\begin{proposition}
$\overline{\mu}_K \le -\beta$.
\end{proposition}

\begin{proof}
\begin{eqnarray}
\overline{\mu}_K &=&\sum_{i=1}^K {\log_K q_i\over K}  \nonumber  \\
&=& { {(1 +\alpha_2 + \dots + \alpha_K)}\over K}  \log_K q_1   \nonumber   \\
&=&     {(1+\alpha_2+\dots+\alpha_K)\over K}  {\log_{R_0} q_1 \over \log_{R_0} K}  \nonumber  \\
&=&     {(1+\alpha_2+\dots+\alpha_K)\over K}  {-\beta \over \log_{R_0} K} \nonumber  \\
&=&     {(1+\alpha_2+\dots+\alpha_K)\over K \log_{R_0} K}  (-\beta).  \nonumber
\end{eqnarray}
\end{proof}
It is now clear that $\overline{\mu}_K\le -\beta$ will hold provided
$\Bigl[ {(1+\alpha_2 +\dots +\alpha_K)\over K\log_{R_0} K}\Bigr]\ge 1$, and
this is true because
\begin{eqnarray}
R_{0}^{ \bigl({1+\alpha_2+\dots+\alpha_K}\bigr)/K}&=&\bigl(R_{0}^1 R_{0}^{\alpha_2}\dots R_{0}^{\alpha_K}\bigr)^{1\over K} \nonumber \\
&\ge& K \over R_{0}^{-1} + R_{0}^{-\alpha_2} + \dots +R_{0}^{-\alpha_K}  \nonumber \\
&=&K, \nonumber
\end{eqnarray} 
where we use the geometric-harmonic mean inequality and then the fact that 
$R_{0}^{-1} + R_{0}^{-\alpha_2} + \dots +R_{0}^{-\alpha_K}=1$.
Therefore,  $\Bigl[ {(1+\alpha_2 +\dots +\alpha_K)\over K\log_{R_0} K}\Bigr]\ge 1$, and so 
$\overline{\mu}_K\le -\beta$.   In the special case of Miller's equiprobability model, exact equality holds:
$\overline{\mu}_K=-\beta$. 

We will just remark here that our result can be looked at from a much more general perspective, as will be discussed elsewhere.

\section*{References}

\noindent\hangindent=.7cm
Bell, Timothy  C., Cleary, J.G., and Witten, I. H., 1990.  Text Compression. Prentice-Hall, Englewood Cliffs, N.J.

\noindent\hangindent=.7cm
Bochkarev, V.V. and Lerner, E. Yu., 2012. The Zipf law for random texts with unequal probabilities of occurrence of letters and
the Pascal pyramid.  Russian Math., 56, 25-27.

\noindent\hangindent=.7cm
Brillouin, L., 1956.  Science and Information Theory.  Academic Press, NY, NY.

\noindent\hangindent=.7cm
Conrad, B. and Mitzenmacher, M., 2004.  Power laws for monkeys typing randomly: the case of unequal letter probabilities.  IEEE Trans. Info. Theory, 50(7), 1403-1414. 

\noindent\hangindent=.7cm
Csisz\'ar, I., 1969.  Simple proofs of some theorems on noiseless channels.  Information and Control, 14, 285-298.

\noindent\hangindent=.7cm
Ekstr\"om, M., 1997.  Strong limit theorems for sums of logarithms of $m^{\rm th}$ order spacing estimates.  Research Report No. 8, Dept. of Mathematical Statistics, Umea University.  

\noindent\hangindent=.7cm
Good, I.J., 1969.  Statistics of Language: Introduction. In: A.R. Meethan and R.A. Hudson, editors,   Encyclopedia of Linguistics, Information and Control, Pergamon Press, Oxford, England, pp. 567-581. 

\noindent \hangindent=.7cm
Mandelbrot, B., 1953.  Jeux de communication., Publ. Inst. Stat., Univ. Paris, 2, 1-124.

\noindent\hangindent=.7cm
Mandelbrot, B., 1954a. On recurrent noise limiting coding. In:  Proceedings of the Symposium on Information Networks, Microwave Research Institute, Polytechnic Institute of Brooklyn, 205-221. 

\noindent\hangindent=.7cm
Mandelbrot, B., 1954b. Simple games of strategy occurring in communication through natural languages.  Trans. I.R.E., PGIT-3, 124-137.

\noindent\hangindent=.7cm
Miller, G.A., 1957. Some effects of intermittent silence.  Amer. J. Psych., 70(2), 311-314. 

\noindent\hangindent=.7cm
Miller, G.A. and Chomsky, N.,  1963. Finitary Models of Language Users.  In:  R.D. Luce, R.R. Bush and E. Galanter, editors,  Handbook of Mathematical Psychology, vol. 2,
New York: Wiley, pp. 419-491. 

\noindent\hangindent=.7cm
Mitzenmacher, M., 2004. A brief history of generative models for power law and lognormal distributions. Internet Math., 1(2), 226-251.

\noindent\hangindent=.7cm
Perline, R., 1996.  Zipf's law, the central limit theorem and the random division of the unit interval.  Phy. Rev. E., 54(1), 220-223.

\noindent\hangindent=.7cm
Shannon, C. E., 1948. A mathematical theory of communication. Bell System Tech. J., 27(3), 379-423.

\noindent\hangindent=.7cm
Shao, Y. and Hahn, M., 1995. Limit theorems for the logarithms of sample spacings. Probability and Statistics Letters, 24, 121-132.

\noindent\hangindent=.7cm
Zipf, G.K., 1935.  The Psycho-Biology of Language.  Houghton Mifflin, Boston, MA.  

\noindent\hangindent=.7cm
Zipf, G.K., 1949.  Human Behavior and the Principle of Least Effort.  Addison-Wesley, Cambridge, MA.

\vskip .75in

\noindent
{\it Acknowledgments:}  My thanks to PJ Mills, Chris Monroe and Ron Perline for their helpful comments on various parts of this manuscript.  I also
want to express my appreciation to Carol Hutchins and other staff members of the Courant Institute of Mathematical Sciences Library of New York University for their
assistance.

\end{document}